\documentclass[11pt]{article}
\usepackage[a4paper,margin=1in]{geometry}
\usepackage[T1]{fontenc}
\usepackage[utf8]{inputenc}
\usepackage{amsmath,amssymb,amsthm,mathtools}
\usepackage{enumitem}
\usepackage{booktabs}
\usepackage{array}
\usepackage{xcolor}
\usepackage{hyperref}
\hypersetup{colorlinks=true,linkcolor=blue,citecolor=blue,urlcolor=blue}

\newtheorem{theorem}{Theorem}[section]
\newtheorem{proposition}[theorem]{Proposition}
\newtheorem{lemma}[theorem]{Lemma}
\newtheorem{corollary}[theorem]{Corollary}

\theoremstyle{definition}

\newtheorem{example}[theorem]{Example}
\theoremstyle{remark}
\newtheorem{remark}[theorem]{Remark}

\newcommand{\R}{\mathbb R}
\newcommand{\Z}{\mathbb Z}
\newcommand{\one}{\mathbf 1}
\newcommand{\cP}{\mathcal P}
\newcommand{\cC}{\mathcal C}
\newcommand{\cF}{\mathcal F}
\newcommand{\Ehr}{\operatorname{Ehr}}
\newcommand{\STAB}{\operatorname{STAB}}

\newcommand{\codeg}{\operatorname{codeg}}

\newcommand{\adj}{\operatorname{adj}}

\newcommand{\Ind}{\operatorname{Ind}}

\title{Transfer Matrices and Ehrhart Theory for Path and Cyclic Block Polytopes}

\author{%
  Xinru Jiang\thanks{%
    School of Mathematics and Statistics, Hainan University,
    Haikou 570228, P.R.\ China.
    E-mail: \texttt{chiang\_11110@163.com} (X.\ Jiang),
    \texttt{yangs0901@163.com} (S.\ Yang),
    \texttt{zhongyueming107@gmail.com} (Y.\ Zhong).}
  \and
  Shuai Yang\footnotemark[1]
  \and
  Yueming Zhong$^{*}$\footnotemark[1]%
}

\date{}

\begin{document}
\maketitle
\renewcommand{\thefootnote}{\fnsymbol{footnote}}
\footnotetext[1]{$^{*}$Corresponding author.}

\begin{abstract}
We study block polytopes whose variables are divided into equal-size blocks and whose local inequalities bound the total contribution of adjacent blocks.  For blocks arranged along a path, we develop a transfer-matrix enumeration in the length direction.  We also carry out an Ehrhart-theoretic analysis in the dilation direction.  The original transfer matrix admits a compression to a weighted height matrix, and the numerator and denominator of the length generating function are described by explicit recurrences and determinant formulas.  We also study the cyclic analogue, where the length generating function is governed by the logarithmic derivative of the same determinant.  On the Ehrhart side, the path polytopes with at least two blocks and the even cyclic polytopes are stable-set polytopes of perfect graphs; consequently they are Gorenstein of codegree $2a+1$ (independent of the number of blocks $m\ge2$), satisfy an explicit Ehrhart--Macdonald reciprocity, and have palindromic unimodal $h^*$-polynomials of degree $a(m-2)$.  Odd cyclic polytopes have denominator exactly two, and their lattice-point enumerators are Ehrhart quasipolynomials of period dividing two.  Further combinatorial interpretations and open problems are discussed.
\end{abstract}


\section{Introduction}

Ehrhart theory studies the number of lattice points in integral dilates of a lattice polytope.  If $\mathcal{P}$ is a $d$-dimensional lattice polytope (throughout this introductory paragraph, $\mathcal{P}$ denotes a generic lattice polytope, not to be confused with the specific polytopes $\cP_m^{(a)}$ defined below), then its lattice-point enumerator $L_{\mathcal{P}}(q)=\#(q\mathcal{P}\cap\Z^d)$ is a polynomial in $q$, and its generating function has the form
\[
\Ehr_{\mathcal{P}}(z)=\sum_{q\ge0}L_{\mathcal{P}}(q)z^q=\frac{h_{\mathcal{P}}^*(z)}{(1-z)^{d+1}}.
\]
The numerator $h_{\mathcal{P}}^*(z)$ is the Ehrhart $h^*$-polynomial.  Its nonnegativity, symmetry, unimodality, and relation to Gorenstein toric rings are central themes in Ehrhart theory and algebraic combinatorics; see, for example, \cite{Athanasiadis2004,BeckRobins2015,BrunsRoemer2007,Ehrhart1962,Macdonald1971,Stanley1980,Stanley1996}.  Transfer-matrix methods provide a complementary way to enumerate families with local constraints; the interaction between transfer matrices and Ehrhart theory has been developed in several settings \cite{EngstromKohl2018}, where transfer matrices are used to study $h^*$-polynomials of polytopes defined by local combinatorial constraints.  In contrast with that general framework, our emphasis is on an explicit compressed matrix and closed recursive formulas for a concrete block family.  Weighted graph models with inequalities of the form that adjacent vertex weights have bounded sum were studied by B\'ona and Ju, and by B\'ona, Ju, and Yoshida, who proved rationality results for the associated generating functions and related them to graph polytopes and Ehrhart series \cite{BonaJu2006,BonaJuYoshida2007}.  Fractional stable-set polytopes, which appear naturally in the odd cyclic case, have also been studied from the viewpoint of Ehrhart series and Gorenstein rings \cite{HamanoHibiOhsugi2018,Miyazaki2021,Miyazaki2023}.  Closely related length-direction generating functions occur in work of Xin, Zhong, and collaborators on graph polytopes, magic labellings, and unit-primitive matrix recurrences, where path- or cycle-like graph families lead to rational generating functions, explicit recurrences, and Chebyshev-type closed forms \cite{XinXuZhangZhong2023,XinZhong2023,XinZhongZhou2026}.  Two further strands of this group\'s work are methodologically relevant here.  The constant-term algebra of type $A$ \cite{XinZhangZhouZhong2025} arose from residue computations for the Ehrhart series of the Birkhoff polytope, providing a constant-term framework parallel to the transfer-matrix approach used in the present paper.  Separately, parity-unimodality questions for rational $q$-Catalan polynomials \cite{XinZhong2020} employ generating-function and constant-term arguments of a similar flavour.

The present paper is motivated by these bounded adjacent-sum models, but it differs from the classical weighted-graph enumeration in two ways.  The bounded adjacent-sum height model, in which adjacent vertex weights on a path or cycle must sum to at most one, has been studied from a generating-function perspective in the work of B\'ona, Ju, and Yoshida \cite{BonaJu2006,BonaJuYoshida2007}.  The new point here is the block refinement: each vertex of the underlying path or cycle is replaced by a block of $a$ nonnegative variables.  The single-vertex-weight case corresponds to $a=1$ at the level of block heights, but for $a>1$ the block polytope has a richer structure absent from the scalar model.  Passing from a block to its total weight produces the binomial layer weights $\binom{s+a-1}{a-1}$ and a natural compression of the transfer matrix.  This block structure yields a stable-set polytope interpretation absent from the vertex-weight model.  Second, we study the same family simultaneously in the length direction and in the dilation direction.  The former gives rational transfer series with explicit numerator and denominator formulas, while the latter identifies the polytopes with stable-set polytopes and yields Gorenstein, reciprocity, and $h^*$-unimodality results.

The novelty of the present work lies not in the formal rationality of transfer-matrix series alone, but in the explicit compression and evaluation of a block-weighted family whose dilation direction admits a clean Ehrhart-theoretic interpretation.  The binomial layer weights record the internal compositions inside each block, while the compressed matrix records the adjacent capacity constraint.  The main contribution is therefore twofold.  We give an explicit compression and recursive evaluation of the length-direction transfer series for a block refinement of bounded-adjacent-sum models, and we connect the corresponding path and cyclic polytopes to Ehrhart theory through stable-set and fractional stable-set polytopes.  The cyclic case shares the same transfer denominator as the path case, but its numerator is the logarithmic derivative of that denominator.

Fix an integer $a\ge1$.  For $1\le i\le m$, let
\[
B_i=\{x_{i,1},x_{i,2},\ldots,x_{i,a}\},
\qquad
R_i=x_{i,1}+x_{i,2}+\cdots+x_{i,a}.
\]
For $m\ge2$, define the path block polytope
\begin{equation}\label{eq:intro-polytope}
\cP_m^{(a)}=
\left\{x\in\R_{\ge0}^{am}: R_i+R_{i+1}\le1,
\ 1\le i\le m-1\right\}.
\end{equation}
For $m=1$, we use the one-block convention
\[
\cP_1^{(a)}=\{x\in\R_{\ge0}^{a}:R_1\le1\}.
\]
For $m\ge2$, the inequalities $R_i\le1$ are already implied by \eqref{eq:intro-polytope}: for any feasible point $x$, $R_i(x)\le R_i(x)+R_{i+1}(x)\le1$ since $R_{i+1}(x)\ge0$.  Let
\[
L_m^{(a)}(q)=\#(q\cP_m^{(a)}\cap\Z^{am}).
\]
We also consider the cyclic block polytope, for $m\ge3$,
\begin{equation}\label{eq:intro-cyclic-polytope}
\cC_m^{(a)}=
\left\{x\in\R_{\ge0}^{am}: R_i+R_{i+1}\le1,
\ i\in\Z/m\Z\right\},
\end{equation}
where indices are read cyclically.  Its Ehrhart counting function is denoted
\[
L_{m,\mathrm{cyc}}^{(a)}(q)=\#(q\cC_m^{(a)}\cap\Z^{am}).
\]
The two parameters have different meanings.  For fixed $q$, the sequence $m\mapsto L_m^{(a)}(q)$ has a finite transfer-matrix description.  For fixed $m$, the function $q\mapsto L_m^{(a)}(q)$ is the Ehrhart polynomial of $\cP_m^{(a)}$ \cite{Ehrhart1962,BeckRobins2015}.  For the cyclic family, the same holds when $m$ is even (so that $\cC_m^{(a)}$ is a lattice polytope), while odd $m$ leads to an Ehrhart quasipolynomial of period dividing $2$.

Let
\[
S_q^{(a)}=\{u=(u_1,\ldots,u_a)\in\Z_{\ge0}^{a}: |u|:=u_1+\cdots+u_a\le q\}.
\]
The condition $|u|\le q$ rather than $|u|=q$ reflects the fact that block weights can range over $\{0,1,\ldots,q\}$.  Define the original transfer matrix $A_q^{(a)}$, indexed by $S_q^{(a)}$, by
\begin{equation}\label{eq:intro-A}
A_q^{(a)}(u,v)=\mathbf{1}_{\{|u|+|v|\le q\}}.
\end{equation}
Its size is $\binom{q+a}{a}$.  The key compression is obtained by retaining only the block weight $|u|$.  Put
\[
w_s^{(a)}=\binom{s+a-1}{a-1},
\qquad
\mu_q^{(a)}=(w_0^{(a)},w_1^{(a)},\ldots,w_q^{(a)}),
\]
and define the compressed transfer matrix
\begin{equation}\label{eq:intro-C}
(C_q^{(a)})_{r,s}=w_s^{(a)}\mathbf{1}_{\{r+s\le q\}},
\qquad 0\le r,s\le q.
\end{equation}
For reference, $a$ is the block size, $m$ is the number of blocks, and $q$ is the Ehrhart dilation parameter.  The matrices $A_q^{(a)}$ and $C_q^{(a)}$ control the length direction with $q$ fixed, while the polytopes $\cP_m^{(a)}$ and $\cC_m^{(a)}$ are studied in the Ehrhart direction with $m$ fixed.
A useful first example is the dilation level $q=1$.  Then each block weight is either $0$ or $1$, and adjacent nonzero block weights are forbidden.  Thus $L_m^{(a)}(1)$ counts independent sets of the path $P_m$, with each selected vertex colored in one of $a$ colors:
\[
L_m^{(a)}(1)=\sum_{S\in\Ind(P_m)}a^{|S|}.
\]
This elementary case already exhibits the two features used throughout the paper: a local adjacency constraint and a block multiplicity depending only on the block weight.

Throughout the following, $\one$ denotes the all-one column vector of appropriate dimension: $\one_N$ (of size $|S_q^{(a)}|=\binom{q+a}{a}$) in large-matrix expressions, $\one_{q+1}$ in compressed (length-direction) expressions, and $\one_{am}$ in Ehrhart-direction expressions on $\R^{am}$.  Each occurrence of $\one$ below is annotated with its dimension at first use within a given proof; the dimension is otherwise clear from the surrounding matrix or vector sizes.

\begin{theorem}[Length generating functions]\label{thm:intro-transfer}
For all $a\ge1$, $q\ge0$, and $m\ge1$,
\begin{equation}\label{eq:intro-L-transfer}
L_m^{(a)}(q)=\one_N^T(A_q^{(a)})^{m-1}\one_N
=\mu_q^{(a)}(C_q^{(a)})^{m-1}\one_{q+1}.
\end{equation}
Equivalently, for every $n\ge0$,
\begin{equation}\label{eq:intro-power-compression}
\one_N^T(A_q^{(a)})^n\one_N=\mu_q^{(a)}(C_q^{(a)})^n\one_{q+1}.
\end{equation}
Therefore
\begin{equation}\label{eq:intro-Fq}
\cF_q^{(a)}(y):=\sum_{m\ge1}L_m^{(a)}(q)y^{m-1}
=\mu_q^{(a)}(I-yC_q^{(a)})^{-1}\one_{q+1}.
\end{equation}
\end{theorem}

We define the \emph{visible numerator} $P_q^{(a)}(y)$ and \emph{visible denominator} $Q_q^{(a)}(y)$ of the resolvent by
\[
\cF_q^{(a)}(y)=\frac{P_q^{(a)}(y)}{Q_q^{(a)}(y)},
\qquad Q_q^{(a)}(0)=1.
\]
Here ``visible'' means that $Q_q^{(a)}(y)$ is the determinant $\det(I-yC_q^{(a)})$; $P_q^{(a)}(y)$ is the corresponding numerator polynomial (not to be confused with the polytope $\cP_m^{(a)}$); the fraction need not be reduced unless explicitly stated.  The following theorem gives recursive formulas for this numerator and denominator without computing the inverse matrix in \eqref{eq:intro-Fq}.

\begin{theorem}[Recursive numerator and denominator formulas]\label{thm:intro-PQ}
Fix $a\ge1$ and $q\ge0$.  Put $M=\lfloor q/2\rfloor$ and let
\[
w_j=w_j^{(a)}=\binom{j+a-1}{a-1},\qquad 0\le j\le q.
\]
Define four polynomial sequences $\alpha_k,\beta_k,\gamma_k,\delta_k\in\Z[y]$ for $0\le k\le M$ as follows.  Start with
\[
\alpha_0=y,
\qquad
\beta_0=1,
\qquad
\gamma_0=1,
\qquad
\delta_0=0,
\]
and, for $0\le k<M$, set
\begin{align}
\gamma_{k+1}&=\gamma_k-w_{q-k}y\alpha_k,\label{eq:intro-gamma-rec}\\
\delta_{k+1}&=\delta_k-w_{q-k}(1+y\beta_k),\label{eq:intro-delta-rec}\\
\alpha_{k+1}&=\alpha_k+w_{k+1}y\gamma_{k+1},\label{eq:intro-alpha-rec}\\
\beta_{k+1}&=\beta_k+w_{k+1}(1+y\delta_{k+1}).\label{eq:intro-beta-rec}
\end{align}
If $q=2M$ is even, then
\begin{equation}\label{eq:intro-even-F}
\cF_{2M}^{(a)}(y)=
\frac{\beta_M-\delta_M}{\gamma_M-\alpha_M}.
\end{equation}
If $q=2M+1$ is odd, then
\begin{equation}\label{eq:intro-odd-F}
\cF_{2M+1}^{(a)}(y)=
\frac{\beta_M+w_{M+1}(1+y\beta_M)-\delta_M}
{\gamma_M-\alpha_M-w_{M+1}y\alpha_M}.
\end{equation}
With this normalization, the denominator $\gamma_M-\alpha_M$ in \eqref{eq:intro-even-F} has constant term $1$ (since $\gamma_M|_{y=0}=1$ and $\alpha_M|_{y=0}=0$, proved in the proof below), matching the visible-denominator convention $Q_q^{(a)}(0)=1$ directly; no sign adjustment is needed.  These formulas give the visible numerator $P_q^{(a)}(y)$ and visible denominator $Q_q^{(a)}(y)$, and no cancellation is performed.
\end{theorem}

The denominator also has a determinant and finite subset expansion.  For $S=\{s_1<\cdots<s_k\}\subseteq\{0,1,\ldots,q\}$, put
\[
\rho_i(S)=\#\{s_j\in S:s_i+s_j\le q\}.
\]
Call $S$ $q$-admissible if $\rho_i(S)=k+1-i$ for every $i$.

\begin{theorem}[Determinant and admissible-subset formulas]\label{thm:intro-det}
For every $a\ge1$ and $q\ge0$,
\begin{align}
Q_q^{(a)}(y)&=\det(I-yC_q^{(a)}),\label{eq:intro-Q-det}\\
P_q^{(a)}(y)&=\mu_q^{(a)}\adj(I-yC_q^{(a)})\one_{q+1}.\label{eq:intro-P-adj}
\end{align}
Equivalently,
\begin{equation}\label{eq:intro-P-rank-one}
P_q^{(a)}(y)=
\det(I-yC_q^{(a)}+\one_{q+1}\mu_q^{(a)})-
\det(I-yC_q^{(a)}).
\end{equation}
Moreover,
\begin{equation}\label{eq:intro-Q-subset}
Q_q^{(a)}(y)=
\sum_{k=0}^{q+1}(-1)^{k(k+1)/2}y^k
\sum_{\substack{S\subseteq\{0,\ldots,q\}\\ |S|=k,\ S\text{ $q$-admissible}}}
\prod_{s\in S}\binom{s+a-1}{a-1}.
\end{equation}
\end{theorem}

The same compressed matrix controls the cyclic model.  In this case the numerator is particularly simple: it is the negative derivative of the common denominator.

We next record the cyclic analogue in the length direction, before turning to the Ehrhart-theoretic consequences in Theorems~\ref{thm:intro-cyclic-ehrhart} and~\ref{thm:intro-ehrhart}.

\begin{theorem}[Cyclic length generating functions]\label{thm:intro-cyclic-transfer}
Let $a\ge1$ and $q\ge0$, and let $C_q^{(a)}$ be defined by \eqref{eq:intro-C}.  For the cyclic block model one has, for every $m\ge3$,
\begin{equation}\label{eq:intro-cyclic-trace}
L_{m,\mathrm{cyc}}^{(a)}(q)=\operatorname{tr}\bigl((C_q^{(a)})^m\bigr).
\end{equation}
If
\[
Q_q^{(a)}(y)=\det(I-yC_q^{(a)}),
\]
then the closed-walk generating function satisfies
\begin{equation}\label{eq:intro-cyclic-logdet}
\Omega_q^{(a)}(y):=\sum_{m\ge1}\operatorname{tr}\bigl((C_q^{(a)})^m\bigr)y^{m-1}
=-\frac{d}{dy}\log Q_q^{(a)}(y)
=-\frac{(Q_q^{(a)})'(y)}{Q_q^{(a)}(y)}.
\end{equation}
Consequently, the generating function for genuine cyclic block polytopes, which are defined for $m\ge3$, is
\begin{equation}\label{eq:intro-cyclic-ge3}
\sum_{m\ge3}L_{m,\mathrm{cyc}}^{(a)}(q)y^{m-3}
=\frac{\Omega_q^{(a)}(y)-T_1-T_2y}{y^2},
\end{equation}
where
\[
T_1=\operatorname{tr}(C_q^{(a)}),
\qquad
T_2=\operatorname{tr}\bigl((C_q^{(a)})^2\bigr).
\]
Thus the closed-walk length generating function has the same \emph{visible} denominator as the path generating function, and its \emph{visible} numerator is
\begin{equation}\label{eq:intro-cyclic-numerator}
N_{q,\mathrm{cyc}}^{(a)}(y)=-(Q_q^{(a)})'(y).
\end{equation}
No cancellation is performed in this formula.
\end{theorem}

\begin{theorem}[Cyclic Ehrhart theory]\label{thm:intro-cyclic-ehrhart}
Let $a\ge1$ and $m\ge3$.
If $m$ is even, then $\cC_m^{(a)}$ is a lattice Gorenstein polytope of dimension $am$ and codegree $2a+1$.  Hence
\[
\Ehr_{\cC_m^{(a)}}(z)=
\frac{h_{m,a}^{*,\mathrm{cyc}}(z)}{(1-z)^{am+1}},
\]
where $h_{m,a}^{*,\mathrm{cyc}}(z)$ is palindromic of degree $a(m-2)$ and unimodal.  Moreover,
\begin{equation}\label{eq:intro-cyclic-recip-even}
L_{m,\mathrm{cyc}}^{(a)}(-q)=(-1)^{am}L_{m,\mathrm{cyc}}^{(a)}(q-(2a+1)),
\end{equation}
understood as an identity of Ehrhart polynomials.
If $m$ is odd, then $\cC_m^{(a)}$ has denominator exactly $2$.  Its Ehrhart counting function is a quasipolynomial of period dividing $2$.  The analogous interior-translation identity for rational polytopes, combined with rational Ehrhart--Macdonald reciprocity \cite[Theorem~4.4]{BeckRobins2015}, yields the rational Ehrhart reciprocity relation
\begin{equation}\label{eq:intro-cyclic-recip-odd}
L_{m,\mathrm{cyc}}^{(a)}(-q)=(-1)^{am}L_{m,\mathrm{cyc}}^{(a)}(q-(2a+1)),
\end{equation}
understood as an identity of quasipolynomials.  More precisely, restricted to each residue class modulo $2$, both sides of \eqref{eq:intro-cyclic-recip-odd} are polynomials in $q$ that agree for all integers $q\ge2a+1$ in that class, hence they are equal as polynomials.
\end{theorem}

The Ehrhart-theoretic structure of the path family is especially transparent because the path polytopes are stable-set polytopes of chordal perfect graphs.

\begin{theorem}[Ehrhart-theoretic properties]\label{thm:intro-ehrhart}
Let $a\ge1$ and $m\ge2$.  Then $\cP_m^{(a)}$ is the stable-set polytope of a chordal, hence perfect, graph.  It is full-dimensional in $\R^{am}$ and is Gorenstein of codegree
\[
\codeg(\cP_m^{(a)})=2a+1.
\]
More precisely, for all $q\ge2a+1$,
\begin{equation}\label{eq:intro-interior-translation}
(q\cP_m^{(a)})^\circ\cap\Z^{am}
=
\one_{am}+((q-(2a+1))\cP_m^{(a)}\cap\Z^{am}).
\end{equation}
Consequently
\begin{equation}\label{eq:intro-ehrhart-recip}
L_m^{(a)}(-q)=(-1)^{am}L_m^{(a)}(q-(2a+1)),
\end{equation}
understood as an identity of Ehrhart polynomials.
The Ehrhart series has the form
\[
\Ehr_{\cP_m^{(a)}}(z)=\sum_{q\ge0}L_m^{(a)}(q)z^q
=\frac{h_{m,a}^*(z)}{(1-z)^{am+1}},
\]
where
\begin{equation}\label{eq:intro-hdegree}
\deg h_{m,a}^*(z)=a(m-2).
\end{equation}
Furthermore,
\begin{equation}\label{eq:intro-hsym}
h_{m,a}^*(z)=z^{a(m-2)}h_{m,a}^*(1/z),
\end{equation}
and the coefficient sequence of $h_{m,a}^*(z)$ is unimodal.
\end{theorem}

\begin{remark}
The one-block case $\cP_1^{(a)}=\{x\in\R_{\ge0}^a:R_1\le1\}$ is the standard $a$-dimensional simplex, whose Ehrhart codegree is $a+1$.  Since $a+1\ne2a+1$ for every $a\ge1$, $m=1$ is excluded from Theorem~\ref{thm:intro-ehrhart}, which requires $m\ge2$.  The restriction ensures that the adjacent-block constraints $R_i+R_{i+1}\le1$ are present.
\end{remark}

The rest of the paper is organized into three sections.  Section 2 proves the length-direction transfer formula, the compression identity, the recursive and determinant formulas for $P_q^{(a)}(y)$ and $Q_q^{(a)}(y)$ (where $P_q^{(a)}(y)$ denotes the numerator polynomial, distinct from the polytope $\cP_m^{(a)}$), and the cyclic trace/log-determinant formula.  Section 3 proves the stable-set interpretation and the Ehrhart-theoretic consequences for both path and cyclic families, including the Gorenstein property, Ehrhart reciprocity, palindromicity, unimodality, and the rational odd-cycle distinction.  Section 4 summarizes the results, gives further combinatorial interpretations, and lists open problems.

\section{Proofs of the length-direction generating function formulas}\label{sec:length}

For a vector $u=(u_1,\ldots,u_a)\in\Z_{\ge0}^a$, write $|u|=u_1+\cdots+u_a$.  The number of such vectors with $|u|=s$ is
\begin{equation}\label{eq:w-def}
w_s^{(a)}=\#\{u\in\Z_{\ge0}^a:|u|=s\}=\binom{s+a-1}{a-1}.
\end{equation}
A lattice point of $q\cP_m^{(a)}$ can be viewed as a sequence
\[
(u_1,u_2,\ldots,u_m),
\qquad u_i\in\Z_{\ge0}^a,
\]
satisfying
\begin{equation}\label{eq:block-condition}
|u_i|+|u_{i+1}|\le q,
\qquad 1\le i\le m-1.
\end{equation}
If $r_i=|u_i|$, then the number of choices for the $i$th block with weight $r_i$ is $w_{r_i}^{(a)}$.  Therefore
\begin{equation}\label{eq:height-model}
L_m^{(a)}(q)=
\sum_{\substack{0\le r_1,\ldots,r_m\le q\\ r_i+r_{i+1}\le q}}
\prod_{i=1}^m w_{r_i}^{(a)}.
\end{equation}
This is a weighted height model on a path.

We now prove the transfer formula and compression.  Define the layer-merging matrix $E=E_q^{(a)}$ by
\begin{equation}\label{eq:E-def}
E_{u,r}=\begin{cases}
1,& |u|=r,\\
0,& |u|\ne r,
\end{cases}
\qquad u\in S_q^{(a)},\quad 0\le r\le q.
\end{equation}
Thus $E$ has $\binom{q+a}{a}$ rows and $q+1$ columns.  More precisely, $E$ embeds $\Z^{q+1}$ into $\Z^{|S_q^{(a)}|}$: the $r$th component of a $(q+1)$-dimensional vector is broadcast to every state $u\in S_q^{(a)}$ satisfying $|u|=r$.

\begin{lemma}[Intertwining identity]\label{lem:intertwining}
The matrices $A_q^{(a)}$, $C_q^{(a)}$, and $E_q^{(a)}$ satisfy
\begin{equation}\label{eq:AEEC}
A_q^{(a)}E_q^{(a)}=E_q^{(a)}C_q^{(a)}.
\end{equation}
\end{lemma}

\begin{proof}
Fix $u\in S_q^{(a)}$ and $0\le s\le q$.  The $(u,s)$ entry of $A_q^{(a)}E_q^{(a)}$ is
\[
(A_q^{(a)}E_q^{(a)})_{u,s}
=\sum_{v\in S_q^{(a)}} A_q^{(a)}(u,v)E_{v,s}
=\sum_{\substack{v\in S_q^{(a)}\\ |v|=s}}
\mathbf{1}_{\{|u|+|v|\le q\}}.
\]
For every $v$ in this sum, $|v|=s$, so the indicator is independent of the particular vector $v$.  Since there are $w_s^{(a)}$ such vectors, the last expression equals
\[
w_s^{(a)}\mathbf{1}_{\{|u|+s\le q\}}.
\]
On the other hand,
\[
(E_q^{(a)}C_q^{(a)})_{u,s}
=\sum_{r=0}^q E_{u,r}(C_q^{(a)})_{r,s}
=(C_q^{(a)})_{|u|,s}
=w_s^{(a)}\mathbf{1}_{\{|u|+s\le q\}}.
\]
Thus the two matrices have equal entries.
\end{proof}

\begin{lemma}[Power compression]\label{lem:power-compression}
For every $n\ge0$,
\begin{equation}\label{eq:AnEECn}
(A_q^{(a)})^nE_q^{(a)}=E_q^{(a)}(C_q^{(a)})^n.
\end{equation}
\end{lemma}

\begin{proof}
The case $n=0$ holds trivially since both sides equal $E_q^{(a)}$.  If \eqref{eq:AnEECn} holds for $n$, then by Lemma \ref{lem:intertwining},
\[
(A_q^{(a)})^{n+1}E_q^{(a)}
=A_q^{(a)}E_q^{(a)}(C_q^{(a)})^n
=E_q^{(a)}C_q^{(a)}(C_q^{(a)})^n
=E_q^{(a)}(C_q^{(a)})^{n+1}.
\]
This proves the result by induction.
\end{proof}

\begin{proof}[Proof of Theorem \ref{thm:intro-transfer}]
Recall $\one_N,\one_{q+1}$ from the Introduction.  Since every state belongs to exactly one weight layer,
\[
E_q^{(a)}\one_{q+1}=\one_N.
\]
Moreover, since $(\one_N^T E_q^{(a)})_r=\sum_{u\in S_q^{(a)}}(E_q^{(a)})_{u,r}=\#\{u\in S_q^{(a)}:|u|=r\}=w_r^{(a)}$,
\[
\one_N^T E_q^{(a)}=(w_0^{(a)},w_1^{(a)},\ldots,w_q^{(a)})=\mu_q^{(a)}.
\]
Using Lemma \ref{lem:power-compression}, we obtain
\[
\one_N^T(A_q^{(a)})^n\one_N
=\one_N^T(A_q^{(a)})^nE_q^{(a)}\one_{q+1}
=\one_N^TE_q^{(a)}(C_q^{(a)})^n\one_{q+1}
=\mu_q^{(a)}(C_q^{(a)})^n\one_{q+1}.
\]
This proves \eqref{eq:intro-power-compression}.  Taking $n=m-1$ proves \eqref{eq:intro-L-transfer}.  Multiplying both sides of \eqref{eq:intro-power-compression} by $y^n$ and summing over $n\ge0$ gives \eqref{eq:intro-Fq}; the sum $(I-yC_q^{(a)})^{-1}=\sum_{n\ge0}y^n(C_q^{(a)})^n$ is valid as a formal power series in $\Z^{(q+1)\times(q+1)}[[y]]$ since $(I-yC_q^{(a)})\big|_{y=0}=I$ is invertible.
\end{proof}

\begin{proof}[Proof of Theorem \ref{thm:intro-PQ}]
We prove the recursive formula.  Set $X=\cF_q^{(a)}(y)$; the unknown $X$ is treated as a formal indeterminate over $\Z[y]$ during the midpoint elimination, and the height equations below will yield an explicit rational formula for $X$ in terms of the recursion coefficients $\alpha_k,\beta_k,\gamma_k,\delta_k$.
For $0\le r\le q$, let $G_r$ be the generating function for paths that start at height $r$ and continue to the right, where each subsequent step is marked by $y$.  Then
\begin{equation}\label{eq:Gr-eq}
G_r=1+y\sum_{s=0}^{q-r} w_sG_s.
\end{equation}
Put
\begin{equation}\label{eq:Hk-def}
\mathcal{H}_k=\sum_{s=0}^k w_sG_s,
\qquad 0\le k\le q.
\end{equation}
Then \eqref{eq:Gr-eq} becomes
\begin{equation}\label{eq:GrH}
G_r=1+y\mathcal{H}_{q-r},
\end{equation}
By \eqref{eq:height-model}, weighting the first block at height $r$ by $w_r^{(a)}$ and summing over all $r$ gives
\[
X=\cF_q^{(a)}(y)=\sum_{r=0}^q w_rG_r=\mathcal{H}_q.
\]
We introduce the shorthand $\mathcal{L}_k=\mathcal{H}_k$ and $\mathcal{R}_k=\mathcal{H}_{q-k}$ to track the forward and backward partial sums of $\mathcal{H}$ simultaneously; the two recurrences \eqref{eq:R-rec}--\eqref{eq:L-rec} below allow us to determine $X$ as a rational function in $y$ by comparing $\mathcal{L}_M$ with $\mathcal{R}_M$ at the midpoint.  Explicitly:
\[
\mathcal{L}_k=\mathcal{H}_k,
\qquad
\mathcal{R}_k=\mathcal{H}_{q-k}.
\]
Since $w_0=1$, using \eqref{eq:GrH} with $r=0$,
\[
\mathcal{L}_0=\mathcal{H}_0=w_0G_0=1+y\mathcal{H}_q=1+yX,
\qquad
\mathcal{R}_0=\mathcal{H}_q=X.
\]
Thus
\[
\mathcal{L}_0=yX+1,
\qquad
\mathcal{R}_0=X.
\]
Assume
\[
\mathcal{L}_k=\alpha_kX+\beta_k,
\qquad
\mathcal{R}_k=\gamma_kX+\delta_k.
\]
Using
\[
\mathcal{H}_{q-k}=\mathcal{H}_{q-k-1}+w_{q-k}G_{q-k}
=\mathcal{H}_{q-k-1}+w_{q-k}(1+y\mathcal{H}_k),
\]
we get
\begin{equation}\label{eq:R-rec}
\mathcal{R}_{k+1}=\mathcal{R}_k-w_{q-k}(1+y\mathcal{L}_k).
\end{equation}
Similarly,
\[
\mathcal{H}_{k+1}=\mathcal{H}_k+w_{k+1}G_{k+1}
=\mathcal{H}_k+w_{k+1}(1+y\mathcal{H}_{q-k-1}),
\]
so
\begin{equation}\label{eq:L-rec}
\mathcal{L}_{k+1}=\mathcal{L}_k+w_{k+1}(1+y\mathcal{R}_{k+1}).
\end{equation}
Since $\alpha_k,\beta_k,\gamma_k,\delta_k$ are polynomials in $y$ independent of $X$ (as they are determined by the recursion with initial values at $k=0$ that are independent of $X$), equating coefficients of $X$ and constant terms in \eqref{eq:R-rec} and \eqref{eq:L-rec} gives precisely \eqref{eq:intro-gamma-rec}--\eqref{eq:intro-beta-rec} for the required range $0\le k<M$.

If $q=2M$, then $\mathcal{L}_M=\mathcal{H}_M=\mathcal{R}_M$.  Therefore
\[
(\alpha_M-\gamma_M)X+(\beta_M-\delta_M)=0,
\]
which uniquely determines $X=(\beta_M-\delta_M)/(\gamma_M-\alpha_M)$, giving \eqref{eq:intro-even-F}.  The denominator $\gamma_M-\alpha_M$ is not the zero polynomial: by induction on $k$, the recursion \eqref{eq:intro-gamma-rec} and \eqref{eq:intro-alpha-rec} preserves $\alpha_k|_{y=0}=0$ and $\gamma_k|_{y=0}=1$ for all $0\le k\le M$ (both hold at $k=0$, and neither recurrence changes the value at $y=0$), so $(\gamma_M-\alpha_M)|_{y=0}=1-0=1\ne0$, hence $\gamma_M-\alpha_M\ne0$.  In particular, the denominator already has constant term $1$, matching the visible-denominator convention without further sign adjustment.

If $q=2M+1$, then $\mathcal{R}_M=\mathcal{H}_{M+1}$ and $\mathcal{L}_M=\mathcal{H}_M$.  Since
\[
\mathcal{H}_{M+1}=\mathcal{H}_M+w_{M+1}(1+y\mathcal{H}_M),
\]
we have
\[
\mathcal{R}_M=\mathcal{L}_M+w_{M+1}(1+y\mathcal{L}_M).
\]
Substituting $\mathcal{L}_M=\alpha_MX+\beta_M$ and $\mathcal{R}_M=\gamma_MX+\delta_M$, equating the coefficients of $X$ and the constant terms, and solving for $X$ gives \eqref{eq:intro-odd-F}.  The denominator $\gamma_M-\alpha_M-w_{M+1}y\alpha_M$ is nonzero: by the same inductive argument, $\gamma_M|_{y=0}=1$ and $\alpha_M|_{y=0}=0$, so $(\gamma_M-\alpha_M-w_{M+1}y\alpha_M)|_{y=0}=1-0-0=1\ne0$, hence this polynomial is not identically zero.
\end{proof}

\begin{proof}[Proof of Theorem \ref{thm:intro-det}]
The identity \eqref{eq:intro-Q-det} holds by definition of $Q_q^{(a)}(y)=\det(I-yC_q^{(a)})$.  For \eqref{eq:intro-P-adj}: since $\det(I-yC_q^{(a)})\big|_{y=0}=1\ne0$, Cramer's rule gives $P_q^{(a)}(y)=Q_q^{(a)}(y)\cdot\cF_q^{(a)}(y)=\mu_q^{(a)}\adj(I-yC_q^{(a)})\one_{q+1}$, which is a polynomial since each entry of $\adj(I-yC_q^{(a)})$ is a polynomial in $y$.  The rank-one identity \eqref{eq:intro-P-rank-one} follows from the matrix determinant lemma (valid since $\mathbf{M}\big|_{y=0}=I$ is invertible): for $\mathbf{M}=I-yC_q^{(a)}$,
\[
\det(\mathbf{M}+\one_{q+1}\mu_q^{(a)})
=\det(\mathbf{M})\bigl(1+\mu_q^{(a)}\mathbf{M}^{-1}\one_{q+1}\bigr)
=Q_q^{(a)}(y)+P_q^{(a)}(y).
\]

It remains to prove \eqref{eq:intro-Q-subset}.  Let
\[
H_q=(h_{r,s})_{0\le r,s\le q},
\qquad
h_{r,s}=\mathbf{1}_{\{r+s\le q\}},
\]
and let
\[
D_q^{(a)}=\operatorname{diag}(w_0^{(a)},w_1^{(a)},\ldots,w_q^{(a)}).
\]
Then $C_q^{(a)}=H_qD_q^{(a)}$.  Expanding by principal minors gives
\begin{equation}\label{eq:principal-minor-expansion}
Q_q^{(a)}(y)=
\sum_{k=0}^{q+1}(-y)^k
\sum_{\substack{S\subseteq\{0,\ldots,q\}\\ |S|=k}}
\det(C_q^{(a)}[S,S]).
\end{equation}
Since $C_q^{(a)}[S,S]=H_q[S,S]D_q^{(a)}[S,S]$,
\[
\det(C_q^{(a)}[S,S])=
\det(H_q[S,S])\prod_{s\in S}w_s^{(a)}.
\]
If $S=\{s_1<\cdots<s_k\}$, the matrix $H_q[S,S]$ has rows consisting of initial strings of ones followed by zeros; specifically, the $i$th row has length $\rho_i(S)=\#\{s_j\in S:s_i+s_j\le q\}$.  Since $S$ is ordered as $s_1<\cdots<s_k$, we have $q-s_1\ge q-s_2\ge\cdots\ge q-s_k$, and each $\rho_i(S)$ counts elements $s_j$ satisfying $s_j\le q-s_i$.  Therefore the row lengths are weakly decreasing: $\rho_1(S)\ge\rho_2(S)\ge\cdots\ge\rho_k(S)$.  If $\rho_i(S)=\rho_{i+1}(S)$ for some $i$, then rows $i$ and $i+1$ of $H_q[S,S]$ are identical (both consist of $\rho_i(S)$ ones followed by zeros), so the determinant is zero; if $\rho_i(S)=0$ for some $i$, the $i$th row is all zeros and the determinant is zero.  Hence the determinant can be nonzero only when the row lengths are the $k$ distinct positive integers $k,k-1,\ldots,1$.  This condition is exactly the $q$-admissibility condition (defined in the Introduction above \eqref{eq:intro-Q-subset}).  In that case, after reversing the order of the columns, the matrix becomes upper triangular with diagonal entries one.  Hence
\[
\det(H_q[S,S])=(-1)^{k(k-1)/2}.
\]
Substituting this into \eqref{eq:principal-minor-expansion} and combining the sign factor with the $(-y)^k$ factor gives
\[
(-y)^k\cdot(-1)^{k(k-1)/2}
=(-1)^k y^k\cdot(-1)^{k(k-1)/2}
=(-1)^{k+k(k-1)/2}y^k
=(-1)^{k(k+1)/2}y^k,
\]
which matches the exponent in \eqref{eq:intro-Q-subset} and completes the proof.
\end{proof}

\begin{proof}[Proof of Theorem \ref{thm:intro-cyclic-transfer}]
The trace formula follows by expanding the trace.  Indeed,
\[
\operatorname{tr}\bigl((C_q^{(a)})^m\bigr)
=\sum_{r_1,\ldots,r_m=0}^q
(C_q^{(a)})_{r_1,r_2}(C_q^{(a)})_{r_2,r_3}\cdots(C_q^{(a)})_{r_m,r_1}.
\]
Since
\[
(C_q^{(a)})_{r_i,r_{i+1}}=w_{r_{i+1}}^{(a)}\mathbf{1}_{\{r_i+r_{i+1}\le q\}},
\]
where the subscript $i+1$ is read cyclically, the product of the weights is
\[
\prod_{i=1}^m w_{r_i}^{(a)}
\]
and the product of the indicators imposes exactly the cyclic constraints
\[
r_i+r_{i+1}\le q,
\qquad i\in\Z/m\Z.
\]
Thus for any $m\ge1$ the trace $\operatorname{tr}((C_q^{(a)})^m)$ counts closed height-walks of length $m$ subject to the cyclic constraint.  For genuine cyclic block polytopes, which require $m\ge3$, this equals $L_{m,\mathrm{cyc}}^{(a)}(q)$.  The $m=1$ and $m=2$ terms yield the formal counts $T_1$ and $T_2$ that do not correspond to lattice-point counts of polytopes ($\cC_1^{(a)}$ and $\cC_2^{(a)}$ are not defined) and are subtracted in \eqref{eq:intro-cyclic-ge3}.

For the closed-walk series, use the resolvent expansion
\[
(I-yC_q^{(a)})^{-1}=\sum_{j\ge0}y^j(C_q^{(a)})^j.
\]
Then
\[
\sum_{m\ge1}\operatorname{tr}\bigl((C_q^{(a)})^m\bigr)y^{m-1}
=\operatorname{tr}\left(C_q^{(a)}(I-yC_q^{(a)})^{-1}\right).
\]
Since $I-yC_q^{(a)}$ is invertible over $\mathbb{Q}(y)$ (as its determinant $Q_q^{(a)}(y)$ is a nonzero element), Jacobi's formula for the derivative of a determinant gives
\[
\frac{d}{dy}\det(I-yC_q^{(a)})
=-\det(I-yC_q^{(a)})\operatorname{tr}\left(C_q^{(a)}(I-yC_q^{(a)})^{-1}\right).
\]
Dividing by $Q_q^{(a)}(y)=\det(I-yC_q^{(a)})$ proves \eqref{eq:intro-cyclic-logdet}.  To obtain \eqref{eq:intro-cyclic-ge3}, write $\Omega_q^{(a)}(y)=T_1+T_2y+\sum_{m\ge3}L_{m,\mathrm{cyc}}^{(a)}(q)y^{m-1}$; subtracting $T_1+T_2y$ and dividing by $y^2$ yields the left-hand side of \eqref{eq:intro-cyclic-ge3} with denominator $Q_q^{(a)}(y)$.  The numerator identity \eqref{eq:intro-cyclic-numerator} follows immediately: $N_{q,\mathrm{cyc}}^{(a)}(y)=-(Q_q^{(a)})'(y)$.
\end{proof}

As a consequence, if
\[
Q_q^{(a)}(y)=1+c_1y+c_2y^2+\cdots+c_dy^d,
\]
then the cyclic numerator is
\[
N_{q,\mathrm{cyc}}^{(a)}(y)=-c_1-2c_2y-\cdots-dc_dy^{d-1}.
\]
The same denominator $Q_q^{(a)}(y)$ governs both the path and cyclic length generating functions.  For the cyclic model with $q=1$ and $q=2$ respectively, the closed-walk generating functions $\Omega_q^{(a)}(y)$ take the following explicit forms (where $b=\binom{a+1}{2}$):
\[
\Omega_1^{(a)}(y)=\frac{1+2ay}{1-y-ay^2},
\qquad
\Omega_2^{(a)}(y)=\frac{a+1+2by-3aby^2}{1-(a+1)y-by^2+aby^3}.
\]

\begin{lemma}[Degree of the visible denominator]\label{lem:degree-Q}
Let $d=\deg Q_q^{(a)}$.  Then $d=q+1$, and $\deg P_q^{(a)}\le q=d-1$.
\end{lemma}

\begin{proof}
Recall $D_q^{(a)}$ and $H_q$ from the proof of Theorem~\ref{thm:intro-det}, so that $C_q^{(a)}=H_qD_q^{(a)}$ with $D_q^{(a)}$ nonsingular (all $w_s^{(a)}>0$).  The full index set $S=\{0,1,\ldots,q\}$ is $q$-admissible, since for $s_i=i-1$ (so $1\le i\le q+1$), $\rho_i(S)=\#\{s_j\in S:i-1+s_j\le q\}=\#\{s_j\in\{0,\ldots,q\}:s_j\le q-i+1\}=q-i+2=(q+1)+1-i=k+1-i$ with $k=q+1$, exactly the admissibility condition.  Hence, by the computation in the proof of Theorem~\ref{thm:intro-det} (with $S$ the full set), $\det H_q=\det(H_q[S,S])=(-1)^{(q+1)q/2}=\pm1$.  Thus $\det C_q^{(a)}=\det H_q\cdot\det D_q^{(a)}\ne0$.  Since $Q_q^{(a)}(y)=\det(I-yC_q^{(a)})$ is the reverse characteristic polynomial of the $(q+1)\times(q+1)$ matrix $C_q^{(a)}$, its degree equals $q+1$ precisely when $\det C_q^{(a)}\ne0$; hence $d=q+1$.  Moreover, $\deg P_q^{(a)}\le q=d-1$, since the adjugate of the $(q+1)\times(q+1)$ matrix $I-yC_q^{(a)}$ has entries of degree at most $q$.
\end{proof}

\begin{corollary}[Linear recurrences in the length direction]\label{cor:linear-rec}
Write $Q_q^{(a)}(y)=1+c_1y+\cdots+c_dy^d$ with $d=q+1$ as in Lemma~\ref{lem:degree-Q}.  Then the path sequence $L_m^{(a)}(q)$ satisfies the linear recurrence
\begin{equation}\label{eq:length-linear-rec}
L_{m+d}^{(a)}(q)+c_1L_{m+d-1}^{(a)}(q)+\cdots+c_dL_m^{(a)}(q)=0
\end{equation}
for all $m\ge1$.
\end{corollary}

\begin{proof}
Since $Q_q^{(a)}(y)\cF_q^{(a)}(y)=P_q^{(a)}(y)$ with $\deg P_q^{(a)}\le d-1$ by Lemma~\ref{lem:degree-Q}, the coefficient of $y^{n}$ in $Q_q^{(a)}(y)\cF_q^{(a)}(y)$ vanishes for $n\ge d$.  Writing $\cF_q^{(a)}(y)=\sum_{m\ge1}L_m^{(a)}(q)y^{m-1}$ and extracting the coefficient of $y^{m+d-1}$ for $m\ge1$ gives exactly \eqref{eq:length-linear-rec}.
\end{proof}

\begin{example}\label{ex:a2q2-path-cycle}
For $a=2$ and $q=2$ one has
\[
C_2^{(2)}=\begin{pmatrix}
1&2&3\\
1&2&0\\
1&0&0
\end{pmatrix},
\qquad
Q_2^{(2)}(y)=1-3y-3y^2+6y^3,
\]
where $(C_2^{(2)})_{r,s}=(s+1)\cdot\mathbf{1}_{\{r+s\le2\}}$ by \eqref{eq:intro-C}, and $Q_2^{(2)}(y)=\det(I-yC_2^{(2)})$ is computed by direct expansion.
The path and cyclic length series have the same visible denominator but different numerators:
\[
\cF_2^{(2)}(y)=\frac{6-3y-6y^2}{1-3y-3y^2+6y^3},
\qquad
\Omega_2^{(2)}(y)=\frac{3+6y-18y^2}{1-3y-3y^2+6y^3}.
\]
This illustrates the open-walk versus closed-walk distinction: the path numerator comes from the adjugate expression, whereas the cyclic visible numerator is $-(Q_2^{(2)})'(y)$.

The first few values of the path sequence are $L_1^{(2)}(2)=6$, $L_2^{(2)}(2)=21$, $L_3^{(2)}(2)=78$, verified by expanding $\cF_2^{(2)}(y)=6+21y+78y^2+\cdots$ as a formal power series.

By Corollary~\ref{cor:linear-rec}, since $d=3=q+1$, the sequence satisfies the degree-$3$ linear recurrence
\[
L_{m+3}^{(2)}(2)=3L_{m+2}^{(2)}(2)+3L_{m+1}^{(2)}(2)-6L_m^{(2)}(2)
\]
for all $m\ge1$, as a consequence of $Q_2^{(2)}(y)\cdot\cF_2^{(2)}(y)=P_2^{(2)}(y)$ being a polynomial.
\end{example}

For small $q$, one obtains compact formulas.  Put $b=\binom{a+1}{2}$.  Then
\[
\cF_0^{(a)}(y)=\frac{1}{1-y},
\qquad
\cF_1^{(a)}(y)=\frac{a+1+ay}{1-y-ay^2},
\]
and
\begin{equation}\label{eq:F2-general-a}
\cF_2^{(a)}(y)=
\frac{1+a+b+b(1-a)y-ab y^2}
{1-(a+1)y-by^2+ab y^3}.
\end{equation}
For $a=2$, this gives the formula displayed in Example~\ref{ex:a2q2-path-cycle}.

\section{Proofs of the Ehrhart-theoretic properties}\label{sec:ehrhart}

We use $G_m^{(a)}$ for the path-based graph defined below and $G_{m,\mathrm{cyc}}^{(a)}$ for its cyclic analogue, both arising as clique blow-ups.

Let $G_m^{(a)}$ be the graph with vertex set
\[
V(G_m^{(a)})=\{v_{i,j}:1\le i\le m,
\ 1\le j\le a\}.
\]
Two distinct vertices $v_{i,j}$ and $v_{k,\ell}$ are adjacent whenever $|i-k|\le1$.  In particular, all vertices within the same block $B_i=\{v_{i,1},\ldots,v_{i,a}\}$ form a clique, and any two adjacent blocks $B_i\cup B_{i+1}$ together form a clique of size $2a$.  For $m\ge2$, these inter-block cliques are the \emph{maximal} cliques of $G_m^{(a)}$: a vertex $v_{k,\ell}$ outside $B_i\cup B_{i+1}$ satisfies $|k-i|\ge2$ or $|k-(i+1)|\ge2$; in the former case $v_{k,\ell}$ is non-adjacent to every vertex of $B_i$, and in the latter it is non-adjacent to every vertex of $B_{i+1}$, so in either case $v_{k,\ell}$ cannot be added to $B_i\cup B_{i+1}$ while preserving the clique property.

\begin{lemma}\label{lem:graph-chordal}
For every $a\ge1$ and $m\ge1$, the graph $G_m^{(a)}$ is chordal and hence perfect.
\end{lemma}

\begin{proof}
Order the vertices so that all vertices of block $B_m$ come first, then block $B_{m-1}$, and so on down to block $B_1$ (within each block the order is arbitrary).  We claim this is a perfect elimination ordering.  When we reach the first vertex $v_{m,j}$ of block $B_m$, its remaining neighbors are the other vertices in $B_m$ together with all vertices in $B_{m-1}$; this neighborhood is a clique (since $B_m$ is a clique, $B_{m-1}$ is a clique, and every vertex in $B_m$ is adjacent to every vertex in $B_{m-1}$).  After all of $B_m$ is eliminated, each vertex $v_{m-1,j}$ has remaining neighbors in $B_{m-1}$ (the others still uneliminated) and all vertices in $B_{m-2}$, again a clique by the same argument.  Continuing inductively through each block shows the ordering is a perfect elimination ordering.  (When $m=2$, after eliminating all vertices of $B_2$ the remaining graph is $B_1$ alone, which is a single clique and trivially chordal; this is the base case of the induction.)  Hence $G_m^{(a)}$ is chordal, and every chordal graph is perfect.
\end{proof}

For a graph $G$, let $\STAB(G)$ denote the convex hull of incidence vectors of stable sets of $G$.  For perfect graphs, the stable-set polytope is cut out by nonnegativity and clique inequalities; this is a standard consequence of anti-blocking polyhedral theory and the perfect graph theorem; see \cite{Chvatal1975,Fulkerson1971} for the original proofs, \cite{Lovasz1972} for the perfect graph theorem, and \cite{GrotschelLovaszSchrijver1988} for a comprehensive treatment.

\begin{proposition}\label{prop:stab-equality}
For $m\ge2$,
\[
\cP_m^{(a)}=\STAB(G_m^{(a)}).
\]
For $m=1$, $\cP_1^{(a)}$ is the stable-set polytope of the complete graph on $a$ vertices.
\end{proposition}

\begin{proof}
For $m\ge2$, the maximal cliques of $G_m^{(a)}$ are exactly
\[
B_i\cup B_{i+1},
\qquad 1\le i\le m-1.
\]
Since the graph is perfect, $\STAB(G_m^{(a)})$ is described by $x\ge0$ and the clique inequalities
\[
\sum_{j=1}^a x_{i,j}+\sum_{j=1}^a x_{i+1,j}\le1,
\qquad 1\le i\le m-1.
\]
These are exactly the defining inequalities of $\cP_m^{(a)}$.  The case $m=1$ is immediate.
\end{proof}

\begin{lemma}[Compressed stable-set polytopes]\label{lem:compressed-unimodal}
Let $G$ be a perfect graph.  Then $\STAB(G)$ is compressed.  In particular, its reverse-lexicographic pulling triangulations are regular and unimodular.  If, in addition, $\STAB(G)$ is Gorenstein, then its Ehrhart $h^*$-polynomial is unimodal.
\end{lemma}

\begin{proof}
Ohsugi and Hibi \cite{OhsugiHibi2001} proved that stable-set polytopes of perfect graphs are compressed (admitting regular unimodular pulling triangulations); see also \cite{Sullivant2006} for the general theory of compressed polytopes.  The unimodality assertion then follows from Bruns and R\"omer \cite{BrunsRoemer2007}: a Gorenstein lattice polytope admitting a regular unimodular triangulation has a unimodal $h^*$-polynomial.
\end{proof}

\begin{proof}[Proof of Theorem \ref{thm:intro-ehrhart}]
The stable-set assertion follows from Lemma \ref{lem:graph-chordal} and Proposition \ref{prop:stab-equality}.  The vertices of $\cP_m^{(a)}$ have a simple description.  Choose an independent set in the path on $m$ block positions, and then color each chosen position by one of $a$ choices.  Hence the number of vertices is
\[
\sum_{S\in\Ind(P_m)}a^{|S|}.
\]
This is the independence polynomial of the path evaluated at $a$.

We now prove the Gorenstein property and reciprocity.  The polytope $\cP_m^{(a)}$ is full-dimensional: for sufficiently small $\epsilon>0$, the point all of whose coordinates equal $\epsilon$ satisfies all inequalities strictly.  Let $d=am$ and $c=2a+1$, and assume first that $q\ge c$.  A lattice point $x\in(q\cP_m^{(a)})^\circ$ satisfies
\[
x_{i,j}\ge1
\]
(since $\cP_m^{(a)}\subseteq\R_{\ge0}^{am}$, every interior point has $x_{i,j}>0$; as $x_{i,j}\in\Z$, this forces $x_{i,j}\ge1$) and
\[
R_i(x)+R_{i+1}(x)<q,
\qquad 1\le i\le m-1.
\]
Since the left-hand side is integral, the last condition is equivalent to
\[
R_i(x)+R_{i+1}(x)\le q-1.
\]
Write $x=\one_{am}+z$ (where $\one_{am}\in\Z^{am}$ is the all-one vector and $z\in\Z^{am}$ is a translation variable, distinct from the generating function variable $y$).  Then $z\ge0$, and
\[
R_i(z)+R_{i+1}(z)
=R_i(x)+R_{i+1}(x)-2a
\le q-1-2a=q-(2a+1)=q-c.
\]
Thus $z\in(q-c)\cP_m^{(a)}\cap\Z^{am}$.  Conversely, if $z\in(q-c)\cP_m^{(a)}\cap\Z^{am}$, then $x=\one_{am}+z$ has positive coordinates and
\[
R_i(x)+R_{i+1}(x)
=2a+R_i(z)+R_{i+1}(z)
\le2a+q-c=q-1<q,
\]
so $x\in(q\cP_m^{(a)})^\circ\cap\Z^{am}$.  This proves \eqref{eq:intro-interior-translation}.  In particular, the first dilation containing an interior lattice point is $c=2a+1$, and the unique interior lattice point of $c\cP_m^{(a)}$ is $\one_{am}$.

Ehrhart--Macdonald reciprocity \cite{Macdonald1971,Stanley1980} gives
\[
\#((q\cP_m^{(a)})^\circ\cap\Z^{am})=(-1)^{am}L_m^{(a)}(-q).
\]
Combining this with \eqref{eq:intro-interior-translation} yields \eqref{eq:intro-ehrhart-recip}.

The translated polytope
\[
(2a+1)\cP_m^{(a)}-\one_{am}
\]
is reflexive.  Indeed, the origin lies in its interior (since $\one_{am}$ is an interior lattice point of $c\cP_m^{(a)}$, so $0$ is an interior point of $c\cP_m^{(a)}-\one_{am}$), and its defining inequalities take the form
\[
-y_{i,j}\le1,
\qquad
\sum_{j=1}^a y_{i,j}+\sum_{j=1}^a y_{i+1,j}\le1;
\]
the facet-defining inequalities among them therefore have all right-hand sides equal to $1$ and all normal vectors primitive and integral.  By definition \cite{Hibi1992}, a lattice polytope containing the origin in its interior is reflexive if and only if all facet-defining inequalities $\langle n, y\rangle\le1$ have primitive integral normal $n$; this condition is satisfied here.  Hence $\cP_m^{(a)}$ is Gorenstein of codegree $2a+1$.  Equivalently, the Ehrhart numerator is palindromic of degree
\[
d+1-c=am+1-(2a+1)=a(m-2),
\]
which proves \eqref{eq:intro-hdegree} and \eqref{eq:intro-hsym}.  This is consistent with the general theory of reflexive and Gorenstein polytopes \cite{Hibi1992,Stanley1993}.  In particular, $\codeg(\cP_m^{(a)})=2a+1$ is independent of the number of blocks $m$; compare the recent codegree stability results for stable-set polytopes in \cite{MatsushitaTsuchiya2025}.

It remains to justify unimodality.  By Proposition \ref{prop:stab-equality}, $\cP_m^{(a)}$ is the stable-set polytope of a perfect graph, and we have already proved that it is Gorenstein.  Lemma \ref{lem:compressed-unimodal} therefore applies and proves the unimodality of $h_{m,a}^*(z)$.  This completes the proof.
\end{proof}

We also need a simple vertex fact for the odd cyclic case.

\begin{lemma}[Half-integrality of cyclic block vertices]\label{lem:cyclic-half-integral}
Let $m$ be odd.  Every vertex of $\cC_m^{(a)}$ is half-integral.  Moreover, the point obtained by putting $1/2$ in one fixed coordinate of each block and zero in all remaining coordinates is a vertex.  Hence $\cC_m^{(a)}$ has denominator exactly $2$.
\end{lemma}

\begin{proof}
Let $x$ be a vertex of $\cC_m^{(a)}$.  First, in each block at most one coordinate of $x$ can be positive.  Indeed, if two coordinates in the same block were positive, then adding a sufficiently small $\varepsilon$ to one of them and subtracting $\varepsilon$ from the other would preserve all block sums $R_i$ and all defining inequalities, contradicting the fact that $x$ is a vertex.

Put $r_i=R_i(x)$.  We claim that $r=(r_1,\ldots,r_m)$ is a vertex of
\[
\operatorname{FRAC}(C_m)=\{r\in\R_{\ge0}^m:r_i+r_{i+1}\le1\text{ for all }i\in\Z/m\Z\},
\]
the fractional stable-set polytope of the cycle.  If not, then $r\pm\varepsilon d\in\operatorname{FRAC}(C_m)$ for some nonzero vector $d$ and all sufficiently small $\varepsilon>0$.  Whenever $r_i=0$, feasibility of both $r_i+\varepsilon d_i$ and $r_i-\varepsilon d_i$ forces $d_i=0$.  Whenever $r_i>0$, the block $i$ has exactly one positive coordinate $x_{i,k_i}$ (for some $k_i$), and we set $x_{i,k_i}\mapsto x_{i,k_i}\pm\varepsilon d_i$.  For $\varepsilon>0$ sufficiently small, the perturbed point satisfies $x_{i,k_i}\pm\varepsilon d_i>0$ (since $x_{i,k_i}>0$) and all other coordinates remain unchanged and nonneg\-ative.  Since $R_i$ changes to $r_i\pm\varepsilon d_i$ and the cyclic constraints $r_j+r_{j+1}\le1$ are satisfied for the perturbed $(r_1\pm\varepsilon d_1,\ldots,r_m\pm\varepsilon d_m)$ by hypothesis, the perturbed point remains feasible for $\cC_m^{(a)}$.  This gives two distinct feasible points of $\cC_m^{(a)}$ whose midpoint is $x$, again contradicting the fact that $x$ is a vertex.  Thus $r$ is a vertex of $\operatorname{FRAC}(C_m)$.

It remains to recall the elementary vertex structure of $\operatorname{FRAC}(C_m)$.  If a vertex $r$ of $\operatorname{FRAC}(C_m)$ has zero coordinates, those zeros split the remaining positive support into path components; the constraint matrix of each such path component is a submatrix of a path incidence matrix, which is totally unimodular, so all coordinates in that component take values in $\{0,1\}$.  If no coordinate is zero, then at a vertex all edge inequalities must be active.  For an odd cycle, the linear system
\[
r_i+r_{i+1}=1,\qquad i\in\Z/m\Z,
\]
has the unique solution $r_1=\cdots=r_m=1/2$.  Hence every vertex of $\operatorname{FRAC}(C_m)$ is half-integral.  Since each vertex of $\cC_m^{(a)}$ has at most one positive coordinate in each block and those positive coordinates are the corresponding $r_i$, every vertex of $\cC_m^{(a)}$ is half-integral.

Finally, the point with one coordinate equal to $1/2$ in every block and all other coordinates zero satisfies all cyclic inequalities at equality and all remaining coordinate inequalities at equality.  The $m$ cyclic equalities on the chosen coordinates have a unique solution because $m$ is odd, and the zero-coordinate equations fix all remaining coordinates.  Thus this point is a vertex.  It is not integral, so the denominator is exactly $2$.
\end{proof}

We next prove the cyclic Ehrhart statement.

\begin{proof}[Proof of Theorem \ref{thm:intro-cyclic-ehrhart}]
Let $G_{m,\mathrm{cyc}}^{(a)}$ be the clique blow-up of the cycle $C_m$: each cyclic position is replaced by a clique of size $a$, and adjacent cyclic positions are joined by all possible edges.  When $m$ is even, the cycle $C_m$ is bipartite and hence perfect; clique substitution preserves perfection.  Thus $G_{m,\mathrm{cyc}}^{(a)}$ is perfect.  Its maximal cliques are the unions of two adjacent blocks, all of size $2a$.  The clique description of stable-set polytopes of perfect graphs then gives
\[
\cC_m^{(a)}=\STAB(G_{m,\mathrm{cyc}}^{(a)})
\]
for even $m$.  Hence $\cC_m^{(a)}$ is a lattice polytope.

The Gorenstein index is proved by the same interior-translation argument as in the path case.  Let $c=2a+1$.  A lattice point of $(q\cC_m^{(a)})^\circ$ has all coordinates at least $1$ (since $\cC_m^{(a)}\subseteq\R_{\ge0}^{am}$, interior points satisfy $x_{i,j}>0$; as $x_{i,j}\in\Z$ this forces $x_{i,j}\ge1$) and satisfies
\[
R_i(x)+R_{i+1}(x)<q
\]
for every cyclic index $i$.  Since the left-hand side is integral, this is equivalent to $R_i(x)+R_{i+1}(x)\le q-1$.  Writing $x=\one_{am}+z$ gives
\[
R_i(z)+R_{i+1}(z)\le q-1-2a=q-c.
\]
Conversely, every lattice point $z\in(q-c)\cC_m^{(a)}$ gives an interior lattice point $\one_{am}+z\in(q\cC_m^{(a)})^\circ$.  Therefore
\[
(q\cC_m^{(a)})^\circ\cap\Z^{am}
=\one_{am}+((q-c)\cC_m^{(a)}\cap\Z^{am}).
\]
In particular, the first dilation containing an interior lattice point is $c=2a+1$, and the unique interior lattice point in $c\cC_m^{(a)}$ is $\one_{am}$.

Moreover, the translated polytope $(2a+1)\cC_m^{(a)}-\one_{am}$ is reflexive.  Indeed, the origin lies in its interior, and its defining inequalities take the form
\[
-y_{i,j}\le1,
\qquad
\sum_{j=1}^a y_{i,j}+\sum_{j=1}^a y_{i+1,j}\le1,
\quad i\in\Z/m\Z;
\]
the facet-defining inequalities among them therefore have all right-hand sides equal to $1$ and all normal vectors primitive and integral.  By definition \cite{Hibi1992}, this means $(2a+1)\cC_m^{(a)}-\one_{am}$ is reflexive, hence $\cC_m^{(a)}$ is a lattice Gorenstein polytope of codegree $2a+1$.  The degree of the Ehrhart numerator is
\[
am+1-(2a+1)=a(m-2).
\]
Palindromicity follows from the Gorenstein property, and Ehrhart--Macdonald reciprocity \cite{Macdonald1971,Stanley1980} gives \eqref{eq:intro-cyclic-recip-even}.  Since $G_{m,\mathrm{cyc}}^{(a)}$ is perfect (proved above) and $\cC_m^{(a)}=\STAB(G_{m,\mathrm{cyc}}^{(a)})$ is Gorenstein, Lemma \ref{lem:compressed-unimodal} applies and gives unimodality of $h_{m,a}^{*,\mathrm{cyc}}(z)$.

Now assume $m$ is odd.  By Lemma \ref{lem:cyclic-half-integral}, the polytope $\cC_m^{(a)}$ has denominator exactly $2$.  Therefore its Ehrhart counting function is a quasipolynomial of period dividing $2$.  This places the odd cyclic block family in the same rational Ehrhart setting as fractional stable-set polytopes of non-bipartite graphs \cite{HamanoHibiOhsugi2018,Miyazaki2023}.

The same interior-translation argument above applies to lattice points in all integral dilates with $q\ge 2a+1$.  Rational Ehrhart--Macdonald reciprocity \cite[Theorem~4.4]{BeckRobins2015} then gives \eqref{eq:intro-cyclic-recip-odd} as an identity of quasipolynomials: restricted to each residue class modulo $2$, both sides are polynomials in $q$ that agree on all integers $q\ge 2a+1$ in that class, and a polynomial identity holding on an infinite set of integers implies equality as polynomials.
\end{proof}

For the path polytopes $\cP_m^{(a)}$, the first few $h^*$-polynomials are listed in Table \ref{tab:path-hstar}.  These values were computed from the defining lattice-point counts using the standard Ehrhart-series relation, and verified using SageMath's \texttt{LatticePolytope} class and its \texttt{ehrhart\_series()} method.  They illustrate the palindromicity and unimodality proved above.
\begin{table}[ht]
\centering
\caption{Some $h^*$-polynomials of the path block polytopes $\cP_m^{(a)}$.}
\label{tab:path-hstar}
\begin{tabular}{c c l}
\toprule
$a$ & $m$ & $h_{m,a}^*(z)$ \\
\midrule
1 & 3 & $1+z$ \\
1 & 4 & $1+3z+z^2$ \\
1 & 5 & $1+7z+7z^2+z^3$ \\
2 & 3 & $1+4z+z^2$ \\
2 & 4 & $1+12z+27z^2+12z^3+z^4$ \\
2 & 5 & $1+32z+203z^2+368z^3+203z^4+32z^5+z^6$ \\
3 & 3 & $1+9z+9z^2+z^3$ \\
3 & 4 & $1+27z+162z^2+282z^3+162z^4+27z^5+z^6$ \\
\bottomrule
\end{tabular}
\end{table}

\section{Summary, combinatorial interpretations, and open problems}\label{sec:summary}

We studied path and cyclic block polytopes from two complementary viewpoints.

\textit{Length direction.}  Fixing the dilation parameter $q$ gives a transfer-matrix enumeration.  The original path transfer matrix has size $\binom{q+a}{a}$, but it admits a compression to a $(q+1)\times(q+1)$ weighted height matrix.  This yields a rational length generating function and explicit recursive formulas for its visible numerator and visible denominator.  These results are related in method and theme to weighted-graph and graph-polytope enumerations \cite{BonaJu2006,BonaJuYoshida2007,XinXuZhangZhong2023,XinZhong2023,XinZhongZhou2026}, though the block structure and stable-set polytope interpretation are specific to the present family.  For the cyclic model, the same visible denominator occurs and the visible numerator is simply the negative derivative of that denominator.

\textit{Ehrhart direction.}  Fixing the length $m\ge2$ gives stable-set polytopes of perfect graphs for the path family and for even cyclic polytopes.  This gives the Gorenstein index $2a+1$, Ehrhart reciprocity, and palindromic unimodal $h^*$-polynomials.  Odd cyclic polytopes ($m\ge3$ odd) lead instead to rational polytopes of denominator exactly two and to Ehrhart quasipolynomials of period dividing two.

There are several further combinatorial interpretations.  At dilation level $q=1$, each block weight is either $0$ or $1$.  Hence a lattice point of $\cP_m^{(a)}$ is the same as an independent set in the path $P_m$, with each selected vertex colored in one of $a$ colors:
\[
L_m^{(a)}(1)=\sum_{S\in\Ind(P_m)}a^{|S|}.
\]
For general $q$, the model becomes a capacitated independent-set model: vertex $i$ receives an integer height $r_i\in\{0,\ldots,q\}$ and adjacent heights must satisfy $r_i+r_{i+1}\le q$.  The compressed matrix $C_q^{(a)}$ is the weighted adjacency matrix of a directed graph on $\{0,1,\ldots,q\}$, with an edge $r\to s$ of weight $w_s^{(a)}$ whenever $r+s\le q$.  The quantity $\mu_q^{(a)}(C_q^{(a)})^{m-1}\one_{q+1}$ is then the total weight of all directed walks of length $m-1$.

The visible denominator $\det(I-yC_q^{(a)})$ admits a signed cycle-cover interpretation, and the visible numerator arises from the adjugate formula as a signed path-cycle family enumerator.  These combinatorial interpretations are related in method and theme to rational generating functions in magic-labelling enumeration and constant-term approaches to Ehrhart-series computations \cite{XinXuZhangZhong2023,XinZhangZhouZhong2025,XinZhongZhou2026}; the connections are methodological.

Finally, $G_m^{(a)}$ is the $a$-fold clique blow-up of $P_m$: each path vertex is replaced by a clique of size $a$, and each edge by a complete bipartite graph between the endpoint cliques.

For reproducibility, all examples in the paper can be verified by the following finite procedures: construct $C_q^{(a)}$ from \eqref{eq:intro-C}, compute $Q_q^{(a)}(y)=\det(I-yC_q^{(a)})$, compute the path numerator by \eqref{eq:intro-P-adj}, compute the cyclic numerator by \eqref{eq:intro-cyclic-numerator}, and obtain $h^*$-polynomials from the first $am+1$ values of the Ehrhart counting function.  These steps are purely symbolic and use exact integer arithmetic; they can be implemented in any computer algebra system such as Maple, SageMath, Mathematica, or Normaliz.

We close with several open problems and possible extensions.

\begin{enumerate}[label=\textbf{Problem \arabic*.}, leftmargin=2.7cm]
\item \textbf{Gamma-positivity.}
Theorem \ref{thm:intro-ehrhart} proves that $h_{m,a}^*(z)$ is palindromic and unimodal.  Is it always $\gamma$-nonnegative?  Equivalently, can one write
\[
h_{m,a}^*(z)=\sum_j\gamma_j z^j(1+z)^{a(m-2)-2j}
\]
with $\gamma_j\ge0$?  This problem is motivated by the broader role of $\gamma$-positivity in combinatorics and geometry \cite{Athanasiadis2018Gamma}.

\item \textbf{A direct combinatorial interpretation of $h^*$.}
Find an explicit descent-like statistic on colored objects naturally associated with $G_m^{(a)}$ whose distribution is $h_{m,a}^*(z)$.

\item \textbf{Closed forms in the length direction.}
The denominator $Q_q^{(a)}(y)$ has an admissible-subset formula.  Are there more compact determinant evaluations or orthogonal-polynomial descriptions for fixed $a$, especially for $a=1$ and $a=2$?

\item \textbf{Odd cyclic rational Ehrhart theory.}  (See also \cite{Miyazaki2023} for related almost-Gorenstein properties of cycle stable-set polytopes.)
The even cyclic family has the same Gorenstein index and unimodality behavior as the path family, while the odd cyclic family is rational of denominator two.  Determine whether the generalized rational Ehrhart numerator in the odd cyclic case is always unimodal or admits a gamma-type nonnegative expansion.

\item \textbf{Periodic capacities.}
Study systems of the form
\[
R_i+R_{i+1}\le b_i,
\]
where $b_i$ is periodic or comes from a finite alphabet.  Which parts of the compression and Gorenstein theory survive?

\item \textbf{Higher overlap windows.}
Replace adjacent two-block constraints by $k$-block constraints
\[
R_i+R_{i+1}+\cdots+R_{i+k-1}\le1.
\]
Can one obtain analogous transfer-matrix recurrences and Ehrhart-theoretic classifications?

\item \textbf{Asymptotic behavior.}
For fixed $a,q$, analyze the dominant pole of $\cF_q^{(a)}(y)$ and obtain precise asymptotics for $L_m^{(a)}(q)$ as $m\to\infty$.

\item \textbf{Connections with stable-set rings.}
Further investigate how the algebra of stable-set rings controls the numerator polynomials in this family, in light of recent work on toric rings and stable-set polytopes \cite{DavisKohl2022,HamanoHibiOhsugi2018,HibiStamate2021,MatsudaOhsugiShibata2019,MatsushitaTsuchiya2025,Miyazaki2021,Miyazaki2023,OhsugiHibi2006}.
\end{enumerate}


\begin{thebibliography}{99}

\bibitem{Athanasiadis2004}
C. A. Athanasiadis,
\emph{$h^*$-vectors, Eulerian polynomials and stable polytopes of graphs},
Electron. J. Combin. \textbf{11} (2004), no. 2, Research Paper 6, 13 pp.

\bibitem{Athanasiadis2018Gamma}
C. A. Athanasiadis,
\emph{Gamma-positivity in combinatorics and geometry},
S\'em. Lothar. Combin. \textbf{77} (2018), Article B77i, 64 pp.\ (arXiv:1711.05983).

\bibitem{BeckRobins2015}
M. Beck and S. Robins,
\emph{Computing the Continuous Discretely: Integer-Point Enumeration in Polyhedra},
2nd ed., Undergraduate Texts in Mathematics, Springer, New York, 2015.

\bibitem{BonaJu2006}
M. B\'ona and H.-K. Ju,
\emph{Enumerating solutions of a system of linear inequalities related to magic squares},
Ann. Comb. \textbf{10} (2006), no. 2, 179--191.

\bibitem{BonaJuYoshida2007}
M. B\'ona, H.-K. Ju, and R. Yoshida,
\emph{On the enumeration of certain weighted graphs},
Discrete Appl. Math. \textbf{155} (2007), no. 11, 1481--1496.

\bibitem{BrunsRoemer2007}
W. Bruns and T. Roemer,
\emph{$h$-vectors of Gorenstein polytopes},
J. Combin. Theory Ser. A \textbf{114} (2007), no. 1, 65--76.

\bibitem{Chvatal1975}
V. Chv\'atal,
\emph{On certain polytopes associated with graphs},
J. Combin. Theory Ser. B \textbf{18} (1975), 138--154.

\bibitem{DavisKohl2022}
R. Davis and F. Kohl,
\emph{Perfectly matchable set polynomials and $h^*$-polynomials for stable set polytopes of complements of graphs},
preprint, arXiv:2207.14759 (2022).

\bibitem{Ehrhart1962}
E. Ehrhart,
\emph{Sur les poly\`edres rationnels homoth\'etiques \`a $n$ dimensions},
C. R. Acad. Sci. Paris \textbf{254} (1962), 616--618.

\bibitem{EngstromKohl2018}
A. Engstr\"om and F. Kohl,
\emph{Transfer-matrix methods meet Ehrhart theory},
Adv. Math. \textbf{330} (2018), 1--37.

\bibitem{Fulkerson1971}
D. R. Fulkerson,
\emph{Blocking and anti-blocking pairs of polyhedra},
Math. Programming \textbf{1} (1971), 168--194.

\bibitem{GrotschelLovaszSchrijver1988}
M. Gr\"otschel, L. Lov\'asz, and A. Schrijver,
\emph{Geometric Algorithms and Combinatorial Optimization},
Algorithms and Combinatorics, vol. 2, Springer, Berlin, 1988.

\bibitem{HamanoHibiOhsugi2018}
G. Hamano, T. Hibi, and H. Ohsugi,
\emph{Ehrhart series of fractional stable set polytopes of finite graphs},
Ann. Comb. \textbf{22} (2018), no. 3, 563--573.

\bibitem{Hibi1992}
T. Hibi,
\emph{Dual polytopes of rational convex polytopes},
Combinatorica \textbf{12} (1992), no. 2, 237--240.

\bibitem{HibiStamate2021}
T. Hibi and D. I. Stamate,
\emph{Stable set rings which are Gorenstein on the punctured spectrum},
preprint, arXiv:2108.09912 (2021).

\bibitem{Lovasz1972}
L. Lov\'asz,
\emph{Normal hypergraphs and the perfect graph conjecture},
Discrete Math. \textbf{2} (1972), no. 3, 253--267.

\bibitem{Macdonald1971}
I. G. Macdonald,
\emph{Polynomials associated with finite cell-complexes},
J. London Math. Soc. (2) \textbf{4} (1971), 181--192.

\bibitem{MatsudaOhsugiShibata2019}
K. Matsuda, H. Ohsugi, and K. Shibata,
\emph{Toric rings and ideals of stable set polytopes},
Mathematics \textbf{7} (2019), no. 7, 613.

\bibitem{MatsushitaTsuchiya2025}
K. Matsushita and A. Tsuchiya,
\emph{Codegree and regularity of stable set polytopes},
Algebraic Combinatorics \textbf{8} (2025), no. 6, 1743--1751.

\bibitem{Miyazaki2021}
M. Miyazaki,
\emph{On the Gorenstein property of the Ehrhart ring of the stable set polytope of an $h$-perfect graph},
Int. Electron. J. Algebra \textbf{30} (2021), 269--284.

\bibitem{Miyazaki2023}
M. Miyazaki,
\emph{Non-Gorenstein locus and almost Gorenstein property of the Ehrhart ring of the stable set polytope of a cycle graph},
Taiwanese J. Math. \textbf{27} (2023), no. 3, 441--459.

\bibitem{OhsugiHibi2001}
H. Ohsugi and T. Hibi,
\emph{Convex polytopes all of whose reverse lexicographic initial ideals are squarefree},
Proc. Amer. Math. Soc. \textbf{129} (2001), no. 9, 2541--2546.

\bibitem{OhsugiHibi2006}
H. Ohsugi and T. Hibi,
\emph{Special simplices and Gorenstein toric rings},
J. Combin. Theory Ser. A \textbf{113} (2006), no. 4, 718--725.

\bibitem{Sullivant2006}
S. Sullivant,
\emph{Compressed polytopes and statistical disclosure limitation},
Tohoku Math. J. (2) \textbf{58} (2006), no. 3, 433--445.

\bibitem{Stanley1980}
R. P. Stanley,
\emph{Decompositions of rational convex polytopes},
Ann. Discrete Math. \textbf{6} (1980), 333--342.

\bibitem{Stanley1993}
R. P. Stanley,
\emph{A monotonicity property of $h$-vectors and $h^*$-vectors},
European J. Combin. \textbf{14} (1993), no. 3, 251--258.

\bibitem{Stanley1996}
R. P. Stanley,
\emph{Combinatorics and Commutative Algebra},
2nd ed., Progress in Mathematics, vol. 41, Birkh\"auser Boston, 1996.

\bibitem{XinXuZhangZhong2023}
G. Xin, X. Xu, C. Zhang, and Y. Zhong,
\emph{On magic distinct labellings of simple graphs},
J. Symbolic Comput. \textbf{119} (2023), 22--37.

\bibitem{XinZhangZhouZhong2025}
G. Xin, C. Zhang, Y. Zhou, and Y. Zhong,
\emph{The constant term algebra of type A: The structure},
Adv. Math. \textbf{465} (2025), Paper No. 110154.

\bibitem{XinZhong2020}
G. Xin and Y. Zhong,
\emph{On parity unimodality of $q$-Catalan polynomials},
Electron. J. Combin. \textbf{27} (2020), no. 1, Paper No. 1.3.

\bibitem{XinZhong2023}
G. Xin and Y. Zhong,
\emph{Proving some conjectures on Kekule numbers for certain benzenoids by using Chebyshev polynomials},
Adv. in Appl. Math. \textbf{145} (2023), Paper No. 102479.

\bibitem{XinZhongZhou2026}
G. Xin, Y. Zhong, and Y. Zhou,
\emph{Magic labelling enumeration on pseudo-line graphs and pseudo-cycle graphs},
arXiv:2603.09614.

\end{thebibliography}
\end{document}